\documentclass{amsart}
\usepackage{amsmath}
\usepackage{mathrsfs}
\usepackage{hyperref}
\usepackage{a4wide}
\usepackage{tabularx}
\usepackage{graphicx}
\usepackage{amsmath}
\usepackage{amssymb}
\usepackage{amsbsy}
\usepackage{amsgen}
\usepackage{amsopn}
\usepackage{amscd}
\usepackage{amstext}
\usepackage{amsthm}

\theoremstyle{definition}
\newtheorem{definition}{Definition}
\newtheorem{notation}{Notation}

\theoremstyle{plain}
\newtheorem{theorem}{Theorem}
\newtheorem{lemma}{Lemma}
\newtheorem{corollary}{Corollary}
\newtheorem{proposition}{Proposition}
\newtheorem{fact}{Fact}

\theoremstyle{remark}
\newtheorem{remark}{Remark}

\theoremstyle{definition}
\newtheorem{example}{Example}

\newtheorem{conv}{Conventions}

\begin{document}

\title{Lifting link invariants by functors on nanophrases 
}
\author{Tomonori Fukunaga}
\address{Department of Life, Environment and Applied Chemistry,
Fukuoka Institute of Technology, Fukuoka, 811-0295, Japan}
\email{fukunaga@fit.ac.jp}

\author{Noboru Ito}
\address{
National Institute of Technology, Ibaraki College, 866 Nakane, Hitachinaka, Ibaraki, 312-8508, Japan 
}
\email{nito@gm.ibaraki-ct.ac.jp}
\date{January 9, 2024}
\maketitle

\begin{abstract}
Nanophrases have a filtered structure consisting of an infinite number of  categories, and each category has a homotopy structure.  Among these categories, the one that we are most familiar with is the category of links.  
Interestingly, the category in which the Jones polynomial is defined consists of a category that is actually less informative category than the link category, and so is the quandle category. The former is called the pseudolink category and the latter the quasilink category, which are covered by the link category; there is a known filtration: links cover pseudolinks, pseudolinks cover virtual strings, and virtual strings cover free links.  The introduction and groundwork for nanophrases were done  by Turaev around 2005.  
In this paper, we introduce functors (Theorem~\ref{thm_functor}) from general nanophrases to virtual strings/pseudolinks/quasilinks/free links.  These functors are powerful because of the difference between each of two.  To demonstrate the effectiveness of such new construction of   functors, the Jones pseudolink polynomial (Section~\ref{jones-by-word}), which implies the Jones link polynomial, is extended to general filtered  nanophrases using one of the functors (Section~\ref{sec6}).  
In particular, the information level of the nanophrases used in each category is explicitly described in the above construction process.     
\end{abstract}

\section{Introduction}
V.~Turaev introduced the theory of topology of words/phrases \cite{Turaev2007, Turaev2006, Turaev2007Lec}.  The theory is an extension of the link theory, and it induces a category of nanophrases, which is universal for links and plane curves.  In particular, there are the injection maps from the set of links to that of nanophrases with a certain homotopy.  
Similarly, a structure of nanophrases induces other categories which objects and morphisms are described as words/phrases  and their homotopies.  
Links, virtual links, flat virtual links, and plane curves are examples of categories of nanophrases.  In order to develop the theory of words/phrases, here called nanoword/nanophrase theory, a technique is encoding link diagrams using nanophrases.  
In the nanophrase theory of Turaev, the Jones polynomial is redefined by pseudolinks though pseudolinks have much less information than those of links \cite[Section~8]{Turaev2006}.  
In fact, the Jones link polynomial is nothing more than the Jones pseudolink polynomial; there is a projection from the set of links to a set of pseudolinks, where the projection cannot be bijective \cite[Section~7]{Turaev2006}.  

It is quite natural to consider the functor (Section~\ref{sec6}) that appears in the difference between links and pseudolinks.  We show that this idea is effective (Theorem~\ref{thm_functor}) with the concrete example (: the Jones polynomial,  Section~\ref{sec6}).  

Here is a literature review.  It is known that invariants of virtual strings are extended to general objects, called \'{e}tale phrases, corresponding to geometric objects that allow singularities.  It is given by constructing a functor between nanophrases and virtual strings that generalizes invariants consisting of known techniques to invariants of more complicated objects \cite{fukunaga2009homotopy}.  Hence, our goal here is to generalize such a functor of  \cite{fukunaga2009homotopy}.      

In this paper, we introduce a general category of nanophrases and \emph{knotlike homotopy}, which includes the category of links, and define an invariant of more general objects covering links.  We believe that the objects covering links are geometric objects, based on the following circumstantial evidence.  

For example, it is known that pseudolinks or quasilinks are covered by links, and pseudolinks fully describe the Jones polynomial, and the category of quasilinks implies to quandle theory \cite{Turaev2006}.  Further virtual strings are covered by pseudolinks; free links \cite{Manturov2010} are also covered by pseudolinks or virtual strings.  Links themselves are also covered by nanophrases.  Here, it would be desirable to find a good  covering space as a category of nanophrases that would cover links in a way that would be suitable for studying properties of links, but this is an open question.  It is known that Turaev \cite{Turaev2006} has also defined a more general category, the so-called category of \'{e}tale words, a singularity-theoretic generalization of nanophrases that captures  properties of links.

As a start in this direction of nanophrase research, we will show that the range of applications of the celebrated Kauffman bracket polynomial \cite{Kauffman1987} is extended by applying the homotopy structure of the sequence of coverings: nanophrases, links, pseudolinks,  virtual strings, and free links.  In particular, after we construct the Kauffman bracket polynomial of pseudolinks in two ways (one by Turaev; the other by us), we show that it gives invariants of pseudolinks or virtual strings.  With these preparations, we introduce a functor from general nanophrases to virtual strings/pseudolinks, and using it, we extend the Kauffman bracket polynomial of virtual strings/pseudolinks to covering nanophrases.   
In Section~\ref{sec6}, we define a functor in a different way from Theorem~\ref{thm_functor}.  This is to show that in more specific situations, functors are constructed with finer information, and also to show that the construction of Theorem~\ref{thm_functor} is not all.  
We also mention a relationship between our technique and quandle theory that fits Alexander modules.  

  
\section{Turaev theory of words}\label{turaevswords}
In this section, notations and definitions are based on \cite{Turaev2007, Turaev2006} (see also \cite{Turaev2007Lec}).   
\subsection{Nanowords and Nanophrase}   

An {\it alphabet} is a finite set and {\it letters} are its elements.    
For an alphabet $\alpha$, an $\alpha$-alphabet $\mathcal{A}$ is a set where every element $A$ of $\mathcal{A}$ has a projection $|~|:$ $A\mapsto |A|\in\alpha$.  
A {\it word of length $n \ge 1$} in an alphabet $\mathcal{A}$ is a mapping $w:$ $\hat{n} \to {\mathcal A}$, where $\hat{n}$ $=$ $\{ i \in {\mathbb{N}}~|~1 \le i \le n \}$.  Such a word is encoded by the sequence $w(1)w(2) \cdots w(n)$.  By definition, there exists the unique word $\emptyset$ of length $0$ called the {\it empty word}.  A word $w:$ $\hat{n} \to {\mathcal A}$ is a {\it Gauss word} in an alphabet $\mathcal{A}$ if each element of $\mathcal A$ is the image of precisely two elements of $\hat{n}$ or $w$ is $\emptyset$.  A {\it Gauss phrase} in an alphabet $\mathcal{A}$ is a sequence of words $w_1$, $w_2, \dots$, $w_m$ in $\mathcal{A}$ denoted by $w_1|w_2|\cdots|w_m$ such that $w_1 w_2 \cdots w_m$ is a Gauss word in $\mathcal{A}$.  We call $w_{i}$ the $i$th component of the Gauss phrase.  In particular, if a Gauss phrase has only one component, that component is a Gauss word.  A {\it nanoword} $(\mathcal A, w)$ over $\alpha$ is a pair consisting of an $\alpha$-{\it alphabet} $\mathcal A$ and a Gauss word $w$ in the alphabet $\mathcal A$.  For a nanoword $(\mathcal{A}, w_{1}w_{2}\cdots w_{k})$ over $\alpha$ consisting of subwords $w_{i}$ $(1 \le i \le k)$ of $w$, a {\it nanophrase} of length $k \ge 0$ over $\alpha$ is defined as $(\mathcal{A}, w_{1}|w_{2}|\cdots|w_{k})$.   A nanophrase $(\mathcal{A}, w_{1}|w_{2}|\cdots|w_{k})$ is denoted by $w_{1}|w_{2}|\cdots|w_{k}$ simply when no confusion may occur.    
We call $w_{i}$ the $i$th component of the nanophrase.  
A nanoword $w$ over $\alpha$ yields a nanophrase $w$ of length $1$.  We denote a set of nanophrases over $\alpha$ with a fixed involution $\tau : \alpha \rightarrow \alpha$ by $\mathcal{P}(\alpha, \tau)$.

An {\it isomorphism} of $\alpha$-alphabets $\mathcal{A}_1$ and $\mathcal{A}_2$ is a bijection $f:$ $\mathcal{A}_1$ $\to$ $\mathcal{A}_2$ such that $|A|=|f(A)|$ for all $A \in \mathcal{A}_1$.  Two nanophrases $(\mathcal{A}_{1}, p_{1}= w_{1}|w_{2}|\cdots|w_{k})$ and $(\mathcal{A}_{2}, p_{2}= w'_{1}|w'_{2}|\cdots|w'_{k'})$ over $\alpha$ are {\it isomorphic} if $k = k'$ and there exists an isomorphism of $\alpha$-alphabets $f:$ $\mathcal{A}_1$ $\to$ $\mathcal{A}_2$ such that $w'_{i}=fw_{i}$ for every $i$ $\in \{1, 2, \dots, k\}$.  
  
\subsection{Homotopy of nanophrases}
Let $\alpha$ be an alphabet, $\tau:$ $\alpha \to \alpha$ an involution, and $S$ a subset of $\alpha \times  \alpha \times \alpha$.  We call the triple $(\alpha, \tau, S)$ {\it homotopy data}.  
\begin{definition}[$S$-homotopy]
Let $(\alpha, \tau, S)$ be homotopy data.  
Two nanophrases are $S$-homotopic if one nanophrase is changed into the other by a finite sequence of isomorphisms and three type deformations (H1)--(H3), called {\it homotopy moves}, and their inverses.  The relation $S$-homotopy is denoted by $\simeq_{S}$.  

(H1) Replace ($\mathcal{A}$, $xAAy$) with ($\mathcal{A} \setminus \{A\}$, $xy$) for $\mathcal{A}$, and $x$, $y$ are words in $\mathcal{A} \setminus \{A\}$ that possibly include the character $|$ such that $xy$ is a Gauss phrase.  

(H2) Replace $(\mathcal{A}, xAByBAz)$ with $(\mathcal{A} \setminus \{A, B\}, xyz)$ if $A, B \in \mathcal{A}$ with $\tau(|A|)=|B|$ where $x$, $y$, $z$ are words in $\mathcal{A} \setminus \{A, B\}$ that possibly include the character $|$ such that $xyz$ is a Gauss phrase.  

(H3) Replace $(\mathcal{A}, xAByACzBCt)$ with $(\mathcal{A}, xBAyCAzCBt)$ for $(|A|, |B|, |C|)$ $\in S$, where $x$, $y$, $z$, and $t$ are words in $\mathcal{A}$ that possibly include the character $|$ such that $xyzt$ is a Gauss phrase.  
\end{definition}
We denote $\mathcal{P}(\alpha, \tau)/{\simeq_S}$ by $\mathcal{P}(\alpha, \tau, S)$. 

We note Lemma~\ref{lemma1} and Lemma~\ref{abab} from \cite[Lemma 2.1, Lemma 2.2]{Turaev2006} (The following lemmas are those of Gauss phrases).  

\begin{lemma}\label{lemma1}
Let $(\alpha, \tau, S)$ be homotopy data and let $\mathcal A$ be an $\alpha$-alphabet.  
Let $A$, $B$, and $C$ be distinct letters in $\mathcal A$ and let $x$, $y$, $z$, and $t$ be words that possibly include the character $|$ in the alphabet $\mathcal A$ $\setminus$ $\{A, B, C\}$ such that $xyzt$ is a Gauss phrase in this alphabet.  Then, 

\begin{enumerate}
\item[(i)] $(\mathcal{A}, xAByCAzBCt)$ $\simeq_{S}$ $(\mathcal{A}, xBAyACzCBt)$~{\text{for}}~$(|A|, \tau(|B|), |C|) \in S$,  

\item[(ii)] $(\mathcal{A}, xAByCAzCBt)$ $\simeq_{S}$ $(\mathcal{A}, xBAyACzBCt)$~{\text{for}}~$(\tau(|A|), \tau(|B|), |C|) \in S$, 

\item[(iii)] $(\mathcal{A}, xAByACzCBt)$ $\simeq_{S}$ $(\mathcal{A}, xBAyCAzBCt)$~{\text{for}}~$(\tau(|A|), |B|, |C|) \in S$.  
\end{enumerate}
\end{lemma}

\begin{lemma}\label{abab}
Suppose that $S$ $\cap$ $(\alpha \times \{b\} \times \{b\})$ $\neq$ $\emptyset$ for all $b \in \alpha$.  Let $(\mathcal{A}, xAByABz)$ be a nanophrase over $\alpha$ with $|B| = \tau(|A|)$, where $x$, $y$, and $z$ are words that possibly include the character $|$ in the alphabet $\mathcal{A} \setminus \{A, B\}$ such that $xyz$ is a Gauss phrase in this alphabet.  Then, 
$(\mathcal{A}, xAByABz)$ $\simeq_{S}$ $(\mathcal{A} \setminus \{A, B\}, xyz)$.  
\end{lemma}

\begin{definition}[$\nu$-shift]\label{shiftWord}
Let $\alpha$ be an alphabet.  Let $\nu:$ $\alpha$ $\to$ $\alpha$ be an involution.     
The $\nu$-{\emph{shift}} of a nanoword $(\mathcal{A}, w:\hat{n} \to \mathcal{A})$ over $\alpha$ is the nanoword $(\mathcal{A'}, w':\hat{n} \to \mathcal{A'})$ obtained by steps (1)--(3):

(1) Let $\mathcal{A}$ $=$ $(\mathcal{A} \setminus \{A\})$ $\cup$ $\{A_{\nu}\}$, where $A_{\nu}$ is a letter not belonging to $\mathcal{A}$.  

(2) The projection $\mathcal{A}' \to \alpha$ extends the given projection $\mathcal{A} \setminus \{A\} \to \alpha$ by $|A_{\nu}|=\nu(|A|)$.  

(3) The word $w'$ in the alphabet $\mathcal{A}'$ is defined by $w' = xA_{\nu}yA_{\nu}$ for $w = AxAy$. 

We call $\nu$ a {\it shift involution}. 
\end{definition}


\begin{definition}\label{shiftPhrase}
Fix an alphabet $\alpha$ and a shift involution $\nu$ in $\alpha$.  
For $i$ $=$ $1, \dots, k$, the $i$th $\nu$-{\it shift} of a nanophrase $P$ moves the first letter, say $A$, of $w_{i}$ to the end of $w_{i}$, keeping $|A| \in \alpha$ if $A$ appears in $w_{i}$ only once and applying $\nu$ if $A$ appears in $w_{i}$ twice.  All other words in $P$ are preserved.  
\end{definition}
\begin{definition}[$\nu$-permutation]
Fix an alphabet $\alpha$ and a shift involution $\nu$ in $\alpha$.  
For two words $u$ and $v$ on an $\alpha$-alphabet $\mathcal{A}$, we define the projection $\mathcal{A} \to \alpha$ by sending $A \in \mathcal{A}$ to $\nu(|A|) \in \alpha$ if $A$ appears both in $u$ and $v$, and sending $A$ to $|A|$ otherwise.  We denote the set $\mathcal{A}$ with this projection to $\alpha$ by $\mathcal{A}_{u \cap v}$.  For $i$ $=$ $1, \dots, k-1$, a transformation replacing a nanophrase $(\mathcal{A}, w_{1}|w_{2}|\cdots|w_{k})$ with the nanophrase $(\mathcal{A}_{w_{i} \cap w_{i+1}}, w_{1}|w_{2}|\cdots|w_{i-1}|w_{i+1}|w_{i}|w_{i+2}|\cdots|w_{k})$ is called a $\nu$-{\it permutation} of the $i$th and $(i+1)$st words.  The operation is involutive.  
\end{definition}

Let $\mathcal{P}(\alpha,\tau, S, \nu)$ be the set of nanophrases over $\alpha$ quotiented by the equivalence relation generated by $S$-homotopy, $\nu$-shifts and $\nu$-permutations on phrases.    The set $\mathcal{P}(\alpha,\tau, S, \nu)$ has remarkable information as we see some examples (notations and facts) in \cite{Turaev2006} as follows.  

\begin{notation}\label{def_virtual}
Let $\alpha_{*}$  $=$ $\{a_{+}$, $a_{-}$, $b_{+}$, $b_{-}\}$ with the involutions $\tau_{*}:$ $a_{\pm} \mapsto b_{\mp}$ and $\nu_*:$ $a_{\pm} \mapsto b_{\pm}$.  Let $S_{*}$ $=\{(a_{\pm},$ $a_{\pm}, a_{\pm}),$ $(a_{\pm}, a_{\pm}, a_{\mp}),$ $(a_{\mp}, a_{\pm}, a_{\pm}),$ $(b_{\pm}, b_{\pm}, b_{\pm}),$ $(b_{\pm}, b_{\pm}, b_{\mp}),$ $(b_{\mp}, b_{\pm},$ $b_{\pm})\}$.  
\end{notation}
\begin{fact}[Section~6.3 of \cite{Turaev2006}]\label{fact_v}
There exists a canonical bijection between the set of virtual links and $\mathcal{P}(\alpha_{*}, \tau_{*}, S_{*}, \nu_{*})$.  
\end{fact}
By Fact~\ref{fact_v}, nanophrases in $\mathcal{P}(\alpha_{*}, \tau_{*}, S_{*}, \nu_*)$ are identified with the set of {\it virtual links}.  Further, in order to see detailed information of links, Turaev also defined {\it pseudolinks} corresponding to a projection image of \emph{virtual links} as follows. 
\begin{notation}\label{def_p}
Let $\alpha_{1}$ $=$ $\{-1, 1\}$ with involutions $\tau_1:$ $1 \mapsto -1$ and $\nu_1$ $=$ $\operatorname{id}$. Let $S_{1}$ $=$ $\{(1, 1, 1)$, $(1, 1, -1)$, $(-1, 1, 1)$, $(-1, -1,- 1)$, $(-1, -1, 1)$, $(1, -1, -1)\}$.   
Nanophrases in $\mathcal{P}(\alpha_{1}, \tau_{1}, S_{1}, \nu_1)$ are called {\it pseudolinks}.  
\end{notation}

\begin{fact}[Section~7.1 of \cite{Turaev2006}]\label{fact_p}
The projection $\alpha_{*} \to$ $\alpha_{1};$  $a_{+}, b_{+} \mapsto 1$ and $a_{-}, b_{-} \mapsto -1$ induces the surjective mapping ${\mathcal P}(\alpha_{*}, \tau_{*}, S_{*}, \nu_{*})$ $\to$ ${\mathcal P}(\alpha_{1}, \tau_{1}, S_{1}, \nu_1)$.  
\end{fact}


\begin{notation}\label{def_s}
Let $\alpha_{0}$ $=$ $\{a, b\}$ with the involutions $\tau_{0}$ $:$ $a \mapsto b$ and $\nu_{0}$ $:$ $a \mapsto b$. Let $S_{0}$ $=$ $\{(a, a, a)$, $(b, b, b)\}$.    
\end{notation}
\begin{fact}[Sections~3.2 together with Remark~6.3 of \cite{Turaev2006}]\label{fact_s}
The  projection $\alpha_{*} \to$ $\alpha_{0};$ $a_{-}, a_{+} \mapsto a$ and $b_{-}, b_{+} \mapsto b$ induces the surjective mapping ${\mathcal P}(\alpha_{*}, \tau_{*}, S_{*}, \nu_{*})$ $\to$ ${\mathcal P}(\alpha_{0}, \tau_{0}, S_{0}, \nu_0)$.  
\end{fact}
\begin{remark}
Nanophrases in $\mathcal{P}(\alpha_{0}, \tau_{0}, S_{0}, \nu_0)$ correspond to topological objects, which are called \emph{closed virtual strings} \cite{SilverWilliams2006, Turaev2005}.  See also Remark~\ref{s1tos0}.  
\end{remark}
\begin{notation}\label{def_q}
Let $\alpha_{2}$ $=$ $\{c, d\}$ with involutions $\tau_{2}$
$:$ $\operatorname{id}$ and $\nu_{2}$ $:$ $c \mapsto d$. Let $S_{2}$ $=$ $\{(c, c, c)$, $(c, c, d)$, $(d, c, c)$, $(d, d, d)$, $(d, d, c)$, $(c, d, d)\}$.   
Nanophrases in $\mathcal{P}(\alpha_{2}, \tau_{2}, S_{2}, \nu_2)$ are called \emph{quasilinks}.  
\end{notation}
\begin{fact}[Section~7.2 of \cite{Turaev2006}]\label{fact_q}
The  projection $\alpha_{*} \to$ $\alpha_{2};$ $a_{+}, b_{-} \mapsto c$ and $a_{-}, b_{+} \mapsto d$ induces the surjective mapping ${\mathcal P}(\alpha_{*}, \tau_{*}, S_{*}, \nu_{*})$ $\to$ ${\mathcal P}(\alpha_{2}, \tau_{2}, S_{2}, \nu_2)$.  
\end{fact}
\begin{notation}\label{def_G}
Let $\alpha_{G}$ $=$ $\{ a \}$ with involutions $\tau_{G}$
$:$ $\operatorname{id}$ and $\nu_{G}$ $:$ $\operatorname{id}$. Let $S_{G}$ $=$ $\{(a,a,a)\}$.  
\end{notation}
\begin{remark}
The realization of nanophrases in $\mathcal{P}(\alpha_{G}, \tau_{G}, S_{G}, \nu_G)$ through virtual diagrams is provided in \cite{Manturov2010}, referred to as \emph{free links}. 
\end{remark}

We recall definitions of opposite words and a $\nu$-inversion, and notations $\mathcal{A}_{w}$ and $P_{w}$.  

\begin{definition}[opposite word]
An {\it opposite word} is defined by writing the letters of a word $w$ in the opposite order.   
\end{definition}
\begin{definition}[$\nu$-inversion]\label{inversion}
For a word $w$ on $\mathcal{A}$, let $\mathcal{A}_{w}$ be the same alphabet $\mathcal{A}$ with a new projection $|\cdot|_{w}$ to $\alpha$ defined as follows: for $A \in \mathcal{A}$, set $|A|_{w}$ $=$ $\tau(|A|)$ if $A$ occurs in $w$ once, $|A|_{w}$ $=$ $\nu(|A|)$ if $A$ occurs in $w$ twice, and $|A|_{w}$ $=$ $|A|$ otherwise.   
Let $P$ be a nanophrase $(\mathcal{A}, w_1 | w_2 |\cdots| w_k)$.  Then, the $i$th $\nu$-{\emph{inversion}} of $P$ is an operation consisting of replacements as follows: 
\begin{itemize}
\item Replace $w_i$ with the opposite word $(w_i)^-$.  
\item Replace $\mathcal{A}$ with $A_{w_{i}}$.  
\item The other words in $P$ are preserved.  
\end{itemize}
We call this operation a $\nu$-inversion simply when no confusion may occur.  
We denote the quotient of  $\mathcal{P} (\alpha, \tau, S, \nu)$ by the $\nu$-inversions by $\mathcal{P}_u (\alpha, \tau, S, \nu)$.    
\end{definition}
By Definition~\ref{inversion}, we also call an element of $\mathcal{P}_u (\alpha_1, \tau_1, S_1, \nu_1)$ a \emph{pseudolink} when no confusion may occur.
\begin{notation}\label{notation:Pw}
For a given phrase $P$ on an $\alpha$-alphabet $\mathcal{A}$ and a word $w$ on $\mathcal{A}$, $P_{w}$ denotes the same phrase on the $\alpha$-alphabet $\mathcal{A}_{w}$.  
\end{notation}
\section{Functors on nanophrases}\label{sec7}
We introduce a ``\emph{knotlike}" homotopy, which is a generalization of $S_*$-homotopy.  The link category we know equals to that of nanophrases with $S_*$-homotopy.  

\begin{definition}
Let $\alpha$ be an alphabet with involutions $\tau$ and $\nu$. Assume that $\tau$ and $\nu$ satisfy $\tau\nu = \nu\tau$. 
We say $S_{\sharp} \subset \alpha \times \alpha \times \alpha$
is a \emph{knotlike homotopy data}  if $S_{\sharp}$ is given by the following form:
\begin{eqnarray*}
S_{\sharp} = \{ (a,a,a), (a,a,\nu\tau(a)), (a,\nu\tau(a), \nu\tau(a)) \ | \ a \in \alpha \}.
\end{eqnarray*}
We call an $S$-homotopy $\sim_{S}$ a \emph{knotlike homotopy} if the homotopy data $S$ is a knotlike homotopy data.
\end{definition}

\begin{example}\label{knotlike_ex}
$(1)$ \ Consider the alphabet $\alpha_*$ with the involutions $\tau_*$ and $\nu_*$. Then, $S_*$ is a knotlike homotopy data.
Indeed, by definition, $\tau_*\nu_* = \nu_*\tau_*$. Since $\nu_{*}\tau_{*}(a_{\pm}) = a_{\mp}$ and $\nu_{*}\tau_{*}(b_{\pm}) = b_{\mp}$,
\begin{eqnarray*}
S_* &=& \{ (a_{\pm}, a_{\pm}, a_{\pm}), (a_{\pm}, a_{\pm}, a_{\mp}), (a_{\pm}, a_{\mp}, a_{\mp}), 
(b_{\pm}, b_{\pm}, b_{\pm}), (b_{\pm}, b_{\pm}, b_{\mp}), (b_{\pm}, b_{\mp}, b_{\mp}) \} \\
&=& \{ (a,a,a), (a,a,\nu_*\tau_*(a)), (a, \nu_*\tau_*(a), \nu_*\tau_*(a)) \ | \ a \in \alpha_* \}.
\end{eqnarray*}

$(2)$ \ Consider the alphabet $\alpha_0$ with the involutions $\tau_0$ and $\nu_0$. Then, $S_0$ is a knotlike homotopy data.
Indeed, by definition, $\tau_0\nu_0 = \nu_0\tau_0$. Since $\nu_{0}\tau_{0}(a) = a$ and $\nu_{0}\tau_{0}(b) = b$,
\begin{eqnarray*}
S_0 &=& \{ (a, a, a), (b, b, b)\} \\
&=& \{ (a,a,a), (a,a,\nu_0 \tau_0(a)), (a, \nu_0 \tau_0(a), \nu_0 \tau_0(a)) \ | \ a \in \alpha_0\}.
\end{eqnarray*}

$(3)$ \ Consider the alphabet $\alpha_1$ with the involutions $\tau_1$ and $\nu_1$. Then, $S_1$ is a knotlike homotopy data.
Indeed, by definition, $\tau_1\nu_1 = \nu_1\tau_1$. Since $\nu_{1}\tau_{1}(1) = -1$,
\begin{eqnarray*}
S_1 &=& \{ (\pm1, \pm1, \pm1), (\pm1, \pm1, \mp1), (\pm1, \mp1,\mp1)\} \\
&=& \{ (a,a,a), (a,a,\nu_1 \tau_1(a)), (a, \nu_1 \tau_1(a), \nu_1 \tau_1(a)) \ | \ a \in \alpha_1\}.
\end{eqnarray*}

$(4)$ \ Consider the alphabet $\alpha_2$ with the involutions $\tau_2$ and $\nu_2$. Then, $S_2$ is a knotlike homotopy data.
Indeed, by definition, $\tau_2\nu_2 = \nu_2\tau_2$. Since $\nu_{2}\tau_{2}(c) = d$,
\begin{eqnarray*}
S_2 &=& \{ (c, c, c), (c, c, d), (c, d, d), (d,d,d), (d,d,c), (d,c,c)\} \\
&=& \{ (a,a,a), (a,a,\nu_2 \tau_2(a)), (a, \nu_2 \tau_2(a), \nu_2 \tau_2(a)) \ | \ a \in \alpha_2\}.
\end{eqnarray*}

$(5)$ \ Consider the alphabet $\alpha_G$ with the involutions $\tau_G$ and $\nu_G$. 
Then, $S_G$ is a knotlike homotopy data.
Indeed, by definition, $\tau_G\nu_G = \nu_G\tau_G$. Since $\nu_{G}\tau_{G}(a) = a$,
\begin{eqnarray*}
S_G &=& \{ (a, a, a)\} \\
&=& \{ (a,a,a), (a,a,\nu_G \tau_G(a)), (a, \nu_G \tau_G(a), \nu_G \tau_G(a)) \ | \ a \in \alpha_G\}.
\end{eqnarray*}
\end{example}

Let $\alpha$ be an alphabet and let $\tau$ and $\nu$ be involutions of $\alpha$ and let $S_{\sharp}$ be the knotlike homotopy data.
Assume that $\nu\tau(a) = \tau\nu(a)$ for all $a \in \alpha$.
We define functors form $\mathcal{P}(\alpha, \tau, S_{\sharp}, \nu)$ to $\mathcal{P}(\alpha_{\diamond}, \tau_{\diamond}, S_{\diamond}, \nu_{\diamond})$
for all $\diamond \in \{*, 0, 1, 2, G\}$.

Consider the quotient set $\alpha / \sim_{\tau,\nu}$ where $\sim_{\tau,\nu}$ is the
equivalence relation generated by the following relation: $a \sim b$ if and only if $b = \tau(a)$ or $b = \nu(a)$. Fix a complete system of representative $L$ of the quotient set $\alpha/\sim_{\tau,\nu}$. 
We decompose $L$ to the disjoint union $L_* \cup L_0 \cup L_1 \cup L_2 \cup L_G$ where
\begin{eqnarray*}
L_* &=& \{ a \in L \ | \ a \neq \tau(a), a \neq \nu(a), \tau(a) \neq \nu(a) \}, \\
L_0 &=& \{ a \in L \ | \ a \neq \tau(a), a \neq \nu(a), \tau(a) = \nu(a) \}, \\
L_1 &=& \{ a \in L \ | \ a \neq \tau(a), a = \nu(a) \}, \\
L_2 &=& \{ a \in L \ | \ a = \tau(a), a \neq \nu(a) \},
\\
L_G &=& \{ a \in L \ | \ a = \tau(a), a = \nu(a) \}.  
\end{eqnarray*}
Let $\eta_{\diamond} = L_{\diamond} \cup \tau(L_{\diamond}) \cup \nu(L_{\diamond}) \cup \tau\nu(L_{\diamond})$ and
let $\mathcal{A}$ be an $\eta_{\diamond}$-alphabet for $\diamond \in \{*, 0, 1, 2, G\}$.
Note that $L_{\diamond} \cup \tau(L_{\diamond}) \cup \nu(L_{\diamond}) \cup \tau\nu(L_{\diamond})$ is not necessarily disjoint union. For example, $\tau(L_0) = \nu(L_0)$ and $L_0 = \tau\nu(L_0)$,
$L_1 = \nu(L_1)$ and $\tau(L_1)= \tau\nu(L_1)$.
Then, the following mapping $p_{\diamond} : \mathcal{A} \rightarrow \alpha_{\diamond}$ is well-defined for each $\diamond \in \{*, 0, 1, 2, G\}$:
\begin{eqnarray*}
p_*(A) &=&
\begin{cases}
a_+ \ (|A| \in L_*) \\
a_- \ (|A| \in \nu(L_*)) \\
b_+ \ (|A| \in \tau\nu(L_*)) \\
b_- \ (|A| \in \tau(L_*)),
\end{cases}
\quad
p_0(A) =
\begin{cases}
a \ (|A| \in L_0) \\ 
b \ (|A| \in \tau(L_0)),
\end{cases}
\\
p_1(A) &=&
\begin{cases}
+1 \ (|A| \in L_1) \\ 
-1 \ (|A| \in \tau(L_1)),
\end{cases}
\quad
p_2(A) =
\begin{cases}
c \ (|A| \in L_2) \\ 
d \ (|A| \in \nu(L_2)),
\end{cases}
\\
p_G(A) &=& a \ (|A| \in L_G).
\end{eqnarray*}

\begin{definition}\label{DiamondFunctor}
Let $\alpha$ be an alphabet with involutions $\tau$ and $\nu$. Assume that $\tau\nu(a) = \nu\tau(a)$ for all $a \in \alpha$.
For a nanophrase $(\mathcal{A}, P)$ over $\alpha$, we define 
$\mathcal{F}_{\diamond} : \mathcal{P}(\alpha, \tau) \rightarrow \mathcal{P}(\alpha_{\diamond}, \tau_{\diamond})$, $(\mathcal{A}, P) \mapsto (\mathcal{F}_{\diamond}(\mathcal{A}), \mathcal{F}_{\diamond}(P))$ where $\diamond \in \{*,0,1,2,G\}$
as follows: 

\noindent
(Step 1) Remove $A \in \mathcal{A}$ such that $ |A| \not\in \eta_{\diamond}$ from $(\mathcal{A}, P)$.  

\noindent
(Step 2)  Let $(\mathcal{A}^{\prime}, P^{\prime})$ be the nanophrase over $\eta_{\diamond}$ obtained by removing letters from $(\mathcal{A}, P)$ by using (Step 1).
Then, an $\alpha_{\diamond}$-alphabet $\mathcal{F}_{\diamond}(A)$ is defined by $\mathcal{F}_{\diamond}(\mathcal{A}) = \mathcal{A}^{\prime}$ with the projection $p_{\diamond}$. Moreover, $\mathcal{F}_{\diamond}(P)$ is defined by $P^{\prime}$.
\end{definition}
Theorem~\ref{thm_functor} is a universal functor representing Facts~\ref{fact_p}, \ref{fact_s}, and \ref{fact_q}.  
\begin{theorem}\label{thm_functor}
The map $\mathcal{F}_{\diamond} : \mathcal{P}(\alpha, \tau) \rightarrow \mathcal{P}(\alpha_{\diamond}, \tau_{\diamond})$
induces the map $\mathcal{F}_{\diamond \bullet} : \mathcal{P}(\alpha, \tau, S_{\sharp}, \nu,)$ $\rightarrow$ $\mathcal{P}(\alpha_{\diamond}, \tau_{\diamond}, S_{\diamond}, \nu_{\diamond})$ for every $\diamond \in \{*, 0, 1, 2, G\}$.
\end{theorem}
\begin{proof}
We show that $\mathcal{F}_{\diamond}$ is invariant under $S_{\sharp}$-homotopy moves, $\nu$-shift and $\nu$-permutation.
Throughout this proof, $x^{\prime}$ denotes the word by removing characters from $x$ with the projection $p_{\diamond}$ by Definition~\ref{DiamondFunctor}.  

The case of $\diamond = *$.
Consider the first $S_{\sharp}$-homotopy move
\begin{eqnarray*}
P_1 = xAAy \rightarrow xy = P_2.
\end{eqnarray*}
If $|A| \in \eta_{*}$, then $\mathcal{F}_{*}(P_1) = x^{\prime}AAy^{\prime}$ and this is $S_{\sharp}$-homotopic to 
the nanophrase $\mathcal{F}_*(P_2) = x^{\prime}y^{\prime}$. 
If $|A| \not\in \eta_{*}$, then $\mathcal{F}_{*}(P_1) = x^{\prime}y^{\prime} = \mathcal{F}_*(P_2)$. 
Thus, $\mathcal{F}_*$ is invariant under the first $S_{*}$-homotopy move.

Consider the second $S_{\sharp}$-homotopy move
\begin{eqnarray*}
P_1 = xAByBAz \rightarrow xyz = P_2
\end{eqnarray*}
where $|A| = \tau(|B|)$. If $|A| \in L_*$ ($\tau(L_*)$, $\nu(L_*)$, $\tau\nu(L_*)$,~resp.), 
then $|B| \in \tau(L_*)$ ($L_*$, $\tau\nu(L_*)$, $\nu(L_*)$,~resp.).
Thus,
$\mathcal{F}_{*}(P_1) = x^{\prime}ABy^{\prime}BAz^{\prime}$ with the projection $|A| = a_+$ ($b_-$, $a_-$, $b_+$,~resp.) 
and $|B| = b-$ ($a_+$, $b_+$, $a_-$,~resp.). This implies $\mathcal{F}_{*}(P_1)$ is $S_{\sharp}$-homotopic to 
the nanophrase $\mathcal{F}_*(P_2) = x^{\prime}y^{\prime}z^{\prime}$. 
If $|A| \not\in \eta_{*}$, $\tau(|A|)=|B| \not\in \eta_{*}$.  
Then $\mathcal{F}_{*}(P_1) = x^{\prime}y^{\prime}z^{\prime} = \mathcal{F}_*(P_2)$. 
Thus, $\mathcal{F}_*$ is invariant under the second $S_{*}$-homotopy move.

Consider the third $S_{\sharp}$-homotopy move
\begin{eqnarray*}
P_1 = xAByACzBCt \rightarrow xBAyCAzCBt = P_2
\end{eqnarray*}
where $(|A|, |B|, |C|) \in S_{\sharp}$. 
Note that $S_*$ is denoted as $\{ (a,a,a),$ $(a,a,\nu_*\tau_*(a)),$ $(a, \nu_*\tau_*(a), \nu_*\tau_*(a))$ $\ | \ a \in \alpha_* \}$ as in Example~\ref{knotlike_ex}.
Thus, if $(|A|, |B|, |C|)  \in \{ (a,a,a),$ $(a,a,\nu\tau(a)),$ $(a, \nu\tau(a),$ $\nu\tau(a))$ $| \ a \in \eta_* \}$,
then $\mathcal{F}_1(P_1) = x^{\prime}ABy^{\prime}ACz^{\prime}BCt^{\prime}$ and $(|A|, |B|, |C|) \in S_*$ by the definition of $p_* : {\mathcal{A}} \to \alpha_*$ (Example~\ref{knotlike_ex}).  
This implies $\mathcal{F}_{*}(P_1)$ is $S_{*}$-homotopic to the nanophrase $\mathcal{F}_*(P_2) = x^{\prime}BAy^{\prime}CAz^{\prime}CBt^{\prime}$.
If $(|A|, |B|, |C|)  \not\in \{ (a,a,a),$ $(a,a,\nu\tau(a)),$ $(a, \nu\tau(a),$ $\nu\tau(a))$ $| \ a \not\in \eta_* \}$, then 
$(|A|, |B|, |C|)  \in \{ (a,a,a),$ $(a,a,\nu\tau(a)),$ $(a, \nu\tau(a),$ $\nu\tau(a))$ $| \ a \not\in \eta_* \}$, and then $|A|, |B|, |C| \not\in \eta_*$. This implies $\mathcal{F}_*(P_1) = x^{\prime}y^{\prime}z^{\prime}t^{\prime} = \mathcal{F}_*(P_2)$.
Thus, $\mathcal{F}_*$ is invariant under the third $S_{*}$-homotopy move.

Consider the $\nu$-shift
\begin{eqnarray*}
P_1 = AxAy \rightarrow xByB = P_2
\end{eqnarray*}
where $|B| = \nu(|A|)$. If $|A| \in L_*$ ($\tau(L_*)$, $\nu(L_*)$, $\tau\nu(L_*)$,~resp.), 
then $|B| \in \nu(L_*)$ ($\nu\tau(L_*)$, $L_*$, $\tau(L_*)$,~resp.) since $\nu\tau\nu(L_*) = \tau\nu\nu(L_*) = \tau(L_*)$.
Thus,
$\mathcal{F}_{*}(P_1) = Ax^{\prime}Ay^{\prime}$ with the projection $|A| = a_+$ ($b_-$, $b_+$, $a_-$ ,~resp.) and
$\mathcal{F}_{*}(P_2) = x^{\prime}By^{\prime}B$ with the projection $|B| = b_+$ ($a_-$, $a_+$, $b_-$,~resp.).
This implies that $\mathcal{F}_{*}(P_1)$ and $\mathcal{F}_{*}(P_2)$ are related by $\nu_*$-shift moves.
If $|A| \not\in \eta_*$, then $|B| \not\in \eta_*$. This implies that $\mathcal{F}_{*}(P_1) = x^{\prime}y^{\prime} = \mathcal{F}_*(P_2)$. 
Thus, $\mathcal{F}_*$ is invariant under the $\nu$-shift.

Consider the $\nu$-permutation
\begin{eqnarray*}
(\mathcal{A}, P_1 = w_{1}|\cdots|w_{k}) \rightarrow (\mathcal{A}_{w_{i} \cap w_{i+1}}, P_2=w_{1}|\cdots|w_{i-1}|w_{i+1}|w_{i}|w_{i+2}|\cdots|w_{k}).
\end{eqnarray*}
If $|A| \in L_*$ ($\tau(L_*)$, $\nu(L_*)$, $\tau\nu(L_*)$,~resp.), 
then $|B| \in \nu(L_*)$ ($\nu\tau(L_*)$, $L_*$, $\tau(L_*)$,~resp.).
Thus, 
$\mathcal{F}_*(A)_{w_i^{\prime} \cap w_{i+1}^{\prime}} = \mathcal{F}_*(\mathcal{A}_{w_i \cap w_{i+1}})$. 
This implies 
\begin{eqnarray*}
&&(\mathcal{F}_*(\mathcal{A}), \mathcal{F}_*(P_1))\\
&=&(\mathcal{F}_*(\mathcal{A}), w_1^{\prime}|\cdots|w_{i-1}^{\prime}|w_{i}^{\prime}|w_{i+1}^{\prime}|w_{i+2}^{\prime}|\cdots|w_{k}^{\prime})\\
&\rightarrow& (\mathcal{F}_*(\mathcal{A})_{w_{i}^{\prime}\cap w_{i+1}^{\prime}}, w_1^{\prime}| \cdots |w_{i-1}^{\prime}|w_{i+1}^{\prime}|w_{i}^{\prime}|w_{i+2}^{\prime}|\cdots|w_{k}^{\prime})\\
&=& (\mathcal{F}_*(\mathcal{A}_{w_i \cap w_{i+1}}), \mathcal{F}_*(P_2))
\end{eqnarray*}
where the arrow in the third line of the above equation means the $\nu_*$-permutation. 
Thus, $\mathcal{F}_*$ is invariant under the $\nu$-permutation.

By the above, $\mathcal{F}_{*} : \mathcal{P}(\alpha, \tau) \rightarrow \mathcal{P}(\alpha_{*}, \tau_{*})$
induces the map $\mathcal{F}_{* \bullet} : \mathcal{P}(\alpha, \tau, S_{\sharp}, \nu) \rightarrow \mathcal{P}(\alpha_{*}, \tau_{*}, S_{*}, \nu_{*})$.  

The case of $\diamond \in \{0,1,2,G\}$ are proved by a similar argument as above.
\end{proof}
Theorem~\ref{thm_functor} implies corollaries, e.g., including generalization of the Jones polynomial.  
As one of the above functors  $\mathcal{F}_{\diamond}(L)$ implies a covering object of a link $L$.  This allows us to consider the Jones polynomial $J(L)$ as a homotopy invariant of general objects.    
\begin{corollary}\label{coho-alpha_{*}}
The generalized Jones polynomial $J \circ \mathcal{F}_{*}$ is an $S_{\sharp}$-homotopy invariant of nanophrases over an arbitrary $\alpha$.  
\end{corollary}

\begin{remark}
The category $\mathcal{P}(\alpha_{2},{\mathrm{id}}, S_2, \nu_2)$ fits a quandle theory (cf. \cite{Turaev2006}) including the Alexander module.  
\end{remark}

\section{The Jones polynomial for pseudolinks}\label{jones-by-word}
In order to show how useful functors between different categories is, we shall show that it actually generalizes the cerebrated Jones polynomial of links.  

Turaev defined the Jones polynomial of  pseudolinks using recursive
relations (i.e., skein relations) and Fact~\ref{fact_p} \cite[Section~8]{Turaev2006}; here we consider another way to define the Jones polynomial of pseudolinks.   
We redefine the Jones pseudolink polynomial using less information from $\mathcal{P}(\alpha_1, \tau_1, S_1, \nu_1)$ rather than  $\mathcal{P}(\alpha_*, \tau_*, S_*, \nu_*)$ in the construction process.  
This is because we would like to give an example that shows a specific effectiveness of functors between different categories derived from nanophrases; this approach would be a direction Turaev theory was originally intended to take.  

First, we review Turaev version   \cite{Turaev2006} of the Jones pseudolink polynomial.  
\begin{definition}[Turaev, \cite{Turaev2006}]\label{jones-turaev} 
Let $p$ be a nanophrase in $\mathcal{P}(\alpha_*, \tau_*, \nu_*)$.  
Let $\epsilon(A)$ be $1$ if $|A|$ $\in$ $\{a_{+}, b_{+}\}$ and $-1$
 otherwise.  The \emph{bracket polynomial} $\langle p \rangle$ of nanophrases
 over $\alpha_{*}$ is defined by the following relations.
\begin{equation*}
\begin{split}
\langle P_1 | AwAz | P_2 \rangle &= t^{\epsilon(A)} \langle P_1 | w | z | P_2 \rangle + t^{-\epsilon(A)} \langle (P_1 | w^{-}z | P_2)_{w} \rangle, \\
\langle P_1 | Aw | Az | P_2 \rangle &= t^{\epsilon(A)} \langle P_1 | wz | P_2 \rangle + t^{-\epsilon(A)} \langle (P_1 | w^{-}z | P_2)_{w} \rangle, \\
\langle P_1 | \emptyset | P_2 \rangle &= -(t^{2} + t^{-2}) \langle P_1 | P_2 \rangle \quad {\text{if}}~(P_1 | \emptyset | P_2) \neq (\emptyset) \\
\langle \emptyset \rangle &= 1.  
\end{split}
\end{equation*}
These relations correspond to an oriented version of the Kauffman bracket polynomial  \cite{Kauffman1987}.  
For any pseudolink $\tilde{p}$, the bracket polynomial $\langle \tilde{p} \rangle$ is
 defined by $\langle \tilde{p} \rangle$ $:=$ $\langle p \rangle$ where $p$ is any nanophrase over $\alpha_*$ which is mapped to $\tilde{p}$ by the surjective map in Fact~\ref{fact_p}.  
 Note that the bracket polynomial of pseudolinks does not depend on the choice of $p$.  
 The Jones polynomial of pseudolinks is defined by $(-t)^{-3w(p)}\langle \tilde{p} \rangle$ where $w(p)$ $=$ $\sum_A |A|$ and $A$ runs over all letters of $\tilde{p}$.  
\end{definition}

Next, in the pseudolink category, we define a Kauffman state sum of pseudolinks in our perspective via states of Viro's style \cite{Viro2004}.  
To do this, we will define states of nanophrases and give some lemmas.  
Let $\mathcal{P}_u (\alpha, \tau, \nu)$ be the set of nanophrases over $\alpha$ quotiented by $\nu$-shifts, $\nu$-inversions, and $\nu$-permutations.  
\begin{definition}[marker, state]\label{maker_state}
Let $(\mathcal{A}, P)$ be a nanophrase in $\mathcal{P}_u (\alpha_1, \tau_1, \nu_1)$.     
We assign a sign $-1$ or $1$ to each letter $A \in \mathcal{A}$ and call the sign a
 {\it marker} of $A$, denoted by $\text{mark}(A)$.  
We call a map $\text{mark} : \mathcal{A} \rightarrow \{-1,1\}$ a {\it
 state} of $P$ and denote by $s$.  
\end{definition}

For  a nanophrase $P$ in $\mathcal{P}_u (\alpha_1, \tau_1, \nu_1)$ assigned with a state $s$, we consider the following deformation ($\ast$): 

\begin{equation*}
(\ast)  \begin{cases}
          & w_1|\cdots|AxAy|\cdots|w_{k} \to \begin{cases}
          & w_1|\cdots|x|y|\cdots|w_{k}~\text{if}~{\rm mark}(A) = |A|\\
          & (w_1|\cdots|x^{-}y|\cdots|w_{k})_{x}~\text{if}~{\rm mark}(A) = -|A|
       \end{cases}
       \\
          & w_1|\cdots|Ax|Ay|\cdots|w_{k} \to \begin{cases}
          & w_1|\cdots|xy|\cdots|w_{k}~\text{if}~{\rm mark}(A) = |A|\\
          & (w_1|\cdots|x^{-}y|\cdots|w_{k})_{x}~\text{if}~{\rm mark}(A) = -|A|.  
       \end{cases}
       \end{cases}
\end{equation*}

For $P$, the nanophrase $\emptyset|\cdots|\emptyset$ consisting of some empty words is obtained by repeating  deformations ($\ast$) from $P$.  Then, the length of this nanophrase $\emptyset|\cdots|\emptyset$ is denoted by $|s|$. 

We denote a letter $A$ with $|A| = 1$ and ${\rm mark}(A) = 1$ 
(${\rm mark}(A) = -1$,~resp.) by $A_{+}$ ($A_{-}$,~resp.),
and we denote a letter $A$ with $|A|=-1$ and ${\rm mark}(A) = 1$ 
(${\rm mark}(A) = -1$,~resp.) by $\overline{A}_{+}$ 
($\overline{A}_{-}$,~resp.).
\begin{remark}
One also follows the deformation ($\ast$) is by smoothing crossings of virtual link diagrams.  It would be  easier for some readers to understand.  However, here, the argument is completed in nanophrses, i.e., it is described in a way that does not require extra efforts of generating and drawing virtual link diagrams.  The key is that we apply the deformation ($\ast$) without changing the category we consider from the pseudolink  category.    
\end{remark}
\begin{example}
We give some examples of states $s$ and $|s|$.

$(1)$ \ Consider $P$ $=$ $ABAB$ with $|A|$ $=$ $|B|$ $=$
 $1$.  
Let $s$ be a state defined by $\text{mark}(A)$ $=$ $1$ and
 $\text{mark}(B)$ $=$ $-1$. 
Note that $P$ is represented as $A_{+}B_{-}A_{+}B_{-}$. 
We compute $|s|$.
We apply the deformation $(\ast)$ to $P$, then we obtain
\begin{equation*}
\begin{split}
A_{+}B_{-}A_{+}B_{-} &\stackrel{(\ast)}{\rightarrow} B_{-}|B_{-}\\
 &\stackrel{(\ast)}{\rightarrow} \emptyset.  
\end{split}
\end{equation*} 
Thus, we obtain $|s| = 1$.

If we consider a state $s$ of $P$ defined by $\text{mark}(A)$ $=$ $-1$ and $\text{mark}(B)$ $=$ $1$, 
then we have
\begin{equation*}
\begin{split}
A_{-}B_{+}A_{-}B_{+} &\stackrel{(\ast)}{\rightarrow} \overline{B}_{+}\overline{B}_{+}\\ &\stackrel{(\ast)}{\rightarrow} \emptyset.  
\end{split}
\end{equation*}
Thus, we have $|s| =1$.  

(2) Consider $P$ $=$ $ABACBC$ with $|A|=1$, $|B|=-1$, and $|C|=1$.   Let $s$ be a state defined by $\text{mark}(A)=1$, $\text{mark}(B)=1$, and $\text{mark}(C)=1$.  
\begin{equation*}
\begin{split}
A_{+}\overline{B}_{+}A_{+}C_{+}\overline{B}_{+}C_{+} &\stackrel{(\ast)}{\rightarrow} \overline{B}_{+}|C_{+}\overline{B}_{+}C_{+}\\ &\stackrel{(\ast)}{\rightarrow} C_{+}C_{+}\\ &\stackrel{(\ast)}{\rightarrow}\emptyset|\emptyset.  
\end{split}
\end{equation*}
Thus, $|s| = 2$.  

If we consider a state $s$ of $P$ defined by $\text{mark}(A)=-1$, $\text{mark}(B)=-1$, and $\text{mark}(C)=1$, then we have
\begin{equation*}
\begin{split}
A_{-}\overline{B}_{-}A_{-}C_{+}\overline{B}_{-}C_{+} &\stackrel{(\ast)}{\rightarrow}{B_{-}}C_{+}{B_{-}}C_{+}\\ &\stackrel{(\ast)}{\rightarrow}\overline{C}_{+}\overline{C}_{+}\\ &\stackrel{(\ast)}{\rightarrow}\emptyset.  
\end{split}
\end{equation*}
Thus, $|s| = 1$.
\end{example}

On the notation $P_w$ (Notation~\ref{notation:Pw}), we prepare the following lemma for $w=x, y$ or $xy$ for later.

\begin{lemma}\label{sub}
Let $P$ be a nanophrase over $\alpha_1$ with the shift involution $\nu_1$. 
Let $x$ and $y$ be subwords of $P$. Then, we have the following.\\
$(1)$ \ $(P_{x})_{x} = P$ \\
$(2)$ \ $(P_{x})_{y} = P_{xy}$. 
\end{lemma}
\begin{proof}
$(1)$ \ By definition, the nanophrase $P_x$ is obtained from the nanophrase $P$ by changing 
the projection of letters that appear only once in the subword $x$ from $1$ ($-1$,~resp.) to $-1$ ($1$,~resp.).
Thus, the projections of letters in $(P_x)_x$ are no different from $P$.

$(2)$ \ Consider a letter $A$ in $P$. By definition, 
the projection of the letter $A$ in $P$ and that of $P_x$ are
different if and only if the letter $A$ appear the word $x$ just once.
the projection of the letter $A$ in $(P_x)_y$ and that of $P_x$ are
different if and only if the letter $A$ appear the word $y$ just once. 
Hence, the projection of the letter $A$ in $(P_x)_y$ and that of $P$ are
different if and only if the letter $A$ appear the word $xy$ just
once. Therefore, we obtain $(P_{x})_{y} = P_{xy}$.
\end{proof}

\begin{lemma}\label{well-def}
$|s|$ is well defined. In other words, $|s|$ does not depend on the order in which letters are deleted.  
\end{lemma}
\begin{proof}
The form of the phrase after deleting a letter depends only on whether the marker of the letter and the projection of the letter are equal \footnote{For a letter $X$, four cases are $(\operatorname{mark}(X), |X|)$ $=$ (1, 1), (-1, -1), (-1, 1), and (1, -1).  Results of calculations of case $(1, 1)$ ($(-1, 1)$,~resp.) equal those of $(-1, -1)$ ($(1, -1)$,~resp.), i.e., cases $(1,1)$ and $(-1, 1)$ only are sufficient.}. Therefore, it does not lose generality to consider only the case where projections of letters are equal to 1, that is, the case all letters do not have an overline.  
For such cases, we obtain Table \ref{table1}.  
\begin{table}
\caption{All the types of the nanophrases that should be checked are listed in the left column.  For each nanophrase in each line of the left column, the nanophrase is got after we delete letters $A$ and $B$ in the right column.}
{\begin{tabular}{|c|c|}
\hline
cases & deleting $A$ and $B$ \\
\hline
$A_{+}xA_{+}yB_{+}zB_{+}t$ &  $x|z|ty$ \\
$A_{-}xA_{-}yB_{+}zB_{+}t$ &  $(z|tx^{-}y)_x$ \\
$A_{-}xA_{-}yB_{-}zB_{-}t$ &  $((z^{-}tx^{-}y)_x)_t$ \\
$A_{+}xB_{+}yA_{+}zB_{+}t$  &  $yxtz$ \\
$A_{-}xB_{+}yA_{-}zB_{+}t$ & $(z^{-}xty^{-})_{zy}$ \\
$A_{-}xB_{-}yA_{-}zB_{-}t$ &  $(x^{-}z|ty^{-})_{xy}$ \\
$A_{+}xA_{+}y|B_{+}zB_{+}t$ & $ x|y|z|t$ \\
$A_{-}xA_{-}y|B_{+}zB_{+}t$ & $(x^{-}y|z|t)_{xy}$ \\
$A_{-}xA_{-}y|B_{-}zB_{-}t$ & $((x^{-}y|z^{-}t)_x)_z$ \\
$A_{+}xB_{+}y|A_{+}zB_{+}t$ & $yz|tx$ \\
$A_{-}xB_{+}y|A_{-}zB_{+}t$ & $(z^{-}xty^{-})_{yz}$ \\
$A_{-}xB_{-}y|A_{-}zB_{-}t$ &  $((x^{-}z|ty^{-})_{x})_z$  \\
$A_{+}x|A_{+}yB_{+}zB_{+}t$ & $z|txy$ \\
$A_{-}x|A_{-}yB_{+}zB_{+}t$ & $(z|tx^{-}y)_x$ \\
$A_{-}x|A_{-}yB_{-}zB_{-}t$ & $((z^{-}tx^{-}y)_x)_z$ \\
$A_{+}x|A_{+}y|B_{+}zB_{+}t$ & $ xy|z|t$ \\
$A_{-}x|A_{-}y|B_{+}zB_{+}t$ & $(z|tx^{-}y)_x$ \\ 
$A_{+}x|A_{+}y|B_{-}zB_{-}t$ & $xy|z^{-}t$ \\
$A_{-}x|A_{-}y|B_{-}zB_{-}t$ & $((x^{-}y|z^{-}t)_x)_z$ \\
$A_{+}x|B_{+}y|A_{+}zB_{+}t$ & $ txzy$ \\
$A_{+}x|B_{-}y|A_{+}zB_{-}t$ & $(xzy^{-}t)_y$  \\
$A_{-}x|B_{-}y|A_{-}zB_{-}t$ & $((z^{-}xt^{-}y)_x)_y$  \\
$A_{+}x|A_{+}y|B_{+}z|B_{+}t$ & $xy|zt$  \\
$A_{+}x|A_{+}y|B_{-}z|B_{-}t$ & $(xy|z^{-}t)_z$ \\
$A_{-}x|A_{-}y|B_{-}z|B_{-}t$ & $((x^{-}y|z^{-}t)_x)_z$ \\
\hline
\end{tabular}}
\label{table1}
\end{table}

Consider the case $A_{+}xA_{+}yB_{+}zB_{+}t$.  
If we delete $A$ first, then 
\begin{eqnarray*}
A_{+}xA_{+}yB_{+}zB_{+}t &\longrightarrow& x|yB_{+}zB_{+}t \\
                        &\longrightarrow& x|B_{+}zB_{+}ty \\
                        &\longrightarrow& x|z|ty. \\  
\end{eqnarray*}
If we delete $B$ first, then
\begin{eqnarray*}
A_{+}xA_{+}yB_{+}zB_{+}t &\longrightarrow& B_{+}zB_{+}tA_{+}xA_{+}y \\
                        &\longrightarrow& z|tA_{+}xA_{+}y \\
                        &\longrightarrow& z|A_{+}xA_{+}yt \\  
                        &\longrightarrow& z|x|yt.
\end{eqnarray*}  
Here, $x|z|ty$ is identified with $z|x|yt$ up to $\nu$-shifts,
$\nu$-inversions, and $\nu$-permutations.
Thus, the resulting phrase does not depend on the order of deletions of letters.   

Since we can prove independencies of orders of deletions of letters for each case in Table~\ref{table1}, thus $|s|$ does not depend on the order of deletions of letters.

\end{proof}

\begin{definition}
For a nanophrase $(\mathcal{A}, P)$ $\in \mathcal{P}(\alpha_1, \tau_1, \nu_1)$ and a state $s$ of $P$,
 we define $[P|s]$ and $[P]$ $\in {\mathbb{Z}}[t, u, d]$ by 
\begin{align*}
[P|s] &:= t^{\sharp\{ A \in \mathcal{A}~|~\text{mark}(A)=1 \}}u^{\sharp\{ A \in \mathcal{A}~|~\text{mark}(A)=-1 \}}d^{|s|-1}, \\ 
[P] &:= \sum_{s} [P|s].  
\end{align*}
\end{definition}

\begin{proposition}\label{propH2}
The polynomial $[P]$ is invariant under the $S_{1}$-homotopy move (H2) for any $P \in \mathcal{P}_u (\alpha_1, \tau_1, \nu_1)$ if and only if $u = t^{-1}$ and $d = -t^{2}-t^{-2}$.  
\end{proposition}
\begin{remark}
In Proposition~\ref{propH2}, it is not only ``if'' but also ``if and only if'' that is essential.  
\end{remark}
\begin{proof}
By Lemma \ref{well-def}, it is sufficient to show this for if the form of $P = P_{1}|ABxBAy|P_{2}$ with $|A| = \tau_1(|B|)$, where $x$ and $y$ are words not including the character $|$.  

In case $|A|$ $=$ $1$ and $|B|$ $=$ $- 1$, 
by Lemmas~\ref{sub} and \ref{well-def}, we obtain
\begin{eqnarray*} 
[P_{1}|ABxBAy|P_{2}]&=&t[P_{1}|BxB|y|P_{2}]+u[(P_{1}|Bx^{-}By|P_{2})_{x}]\\
                     &=&(t^{2}+tud+u^{2})[(P_{1}|x^{-}|y|P_{2})_{x}]+
                           ut[P_{1}|xy|P_{2}].   
\end{eqnarray*}

In case $|A|$ $=$ $- 1$ and $|B|$ $=$ $1$, 
by Lemmas~\ref{sub} and \ref{well-def}, we obtain

\begin{eqnarray*} 
[P_{1}|ABxBAy|P_{2}]&=&t[(P_{1}|Bx^{-}By|P_{2})_{x}]+u[P_{1}|BxB|y|P_{2}]\\
                     &=&(t^{2}+tud+u^{2})[(P_{1}|x^{-}|y|P_{2})_{x}]+
                           ut[P_{1}|xy|P_{2}].   
\end{eqnarray*}

Thus, in both cases, $[P]$ does not change by the second $S_1$- homotopy move if and only if $t^{2}+tud+u^{2}=0$ and $ut=1$.
In other words, $u=t^{-1}$ and $d=-t^{2}-t^{-2}$.

\end{proof}

Substituting $t^{-1}$ for $u$ and $-t^{2}-t^{-2}$ for $d$, we have \[[P]=\sum_{s}t^{\sigma(s)}(-t^{2}-t^{-2})^{|s|-1},\] 
where $\sigma(s)$ $:=$ $\sharp\{ A \in \mathcal{A}~|~\text{mark}(A)=1 \}$ $-$ $\sharp\{ A \in \mathcal{A}~|~\text{mark}(A)=-1 \}$.  

\begin{remark}
Proposition~\ref{propH2} implies that it follows from the invariance under (H2) to fix the state model \cite{Kauffman1987} of the Jones polynomial. In other words, it has the same structure as the \emph{Kauffman trick} for links (Proposition~\ref{prop_trick}).
\end{remark}
\begin{proposition}\label{prop_trick}
The polynomial $[P]$ is invariant under the $S_{1}$-homotopy move (H3) for any $P \in \mathcal{P}_u (\alpha_1, \tau_1, \nu_1)$.  
\end{proposition}
\begin{proof}
First, we consider the case of 
$(|A|$, $|B|$, $|C|)$ $=$ $(\pm 1,$ $\pm 1$, $\pm 1)$.  Consider the third $S_1$-homotopy move
$$P_{1}|ABxACyBCz|P_{2} \longrightarrow P_{1}|BAxCAyCBz|P_{2}.$$
By Definition \ref{maker_state}, deformation$(\ast)$ and Lemma \ref{sub}, we obtain
\begin{eqnarray*}
[P_{1}|ABxACyBCz|P_{2}] &=& t^{3|A|}[P_{1}|xy|z|P_{2}]\\
                          &+&
(2t^{|A|}+t^{-3|A|}+t^{-|A|}(-t^{2}-t^{-2})
)[(P_{1}|zx^{-}y^{-}|P_{2})_{xy}]\\
                           &+&
t^{|A|}[(P_{1}|x^{-}y|z|P_{2})_x]\\ 
                           &+&
t^{-|A|}[(P_{1}|xy^{-}z|P_{2})_y]\\
&+&
t^{-|A|}[(P_{1}|x^{-}yz|P_{2})_x]
\end{eqnarray*}
and
\begin{eqnarray*}
[P_{1}|BAxCAyCBz|P_{2}] &=& t^{3|A|}[P_{1}|xy|z|P_{2}]\\
                          &+&\!\!
(2t^{|A|}+t^{-3|A|}+t^{-|A|}(-t^{2}-t^{-2})
)[(P_{1}|x^{-}y|z|P_{2})_x]\\
                           &+&
t^{|A|}[(P_{1}|x^{-}y^{-}z|P_{2})_{xy}]\\ 
                           &+&
t^{-|A|}[(P_{1}|xy^{-}z|P_{2})_y]\\
&+&
t^{-|A|}[(P_{1}|z^{-}y^{-}x|P_{2})_{yz}].
\end{eqnarray*}
Note that since $|A| \in \{-1, +1\}$, we obtain
$$2t^{|A|}+t^{-3|A|}+t^{-|A|}(-t^{2}-t^{-2})
  =t^{|A|}.$$
Therefore, 
$[P_{1}|ABxACyBCz|P_{2}]$ is equal to $[P_{1}|BAxCAyCBz|P_{2}]$.\par
Consider the third $S_1$-homotopy move
$$P_{1}|ABx|ACyBCz|P_{2} \longrightarrow P_{1}|BAx|CAyCBz|P_{2}.$$
Then 
\begin{eqnarray*}
[P_{1}|ABx|ACyBCz|P_{2}] &=& t^{3|A|}[P_{1}|xzy|P_{2}]\\
                          &+&
(2t^{|A|}+t^{-3|A|}+t^{-|A|}(-t^{2}-t^{-2})
)[(P_{1}|zx^{-}y^{-}|P_{2})_{xy}]\\
                           &+&
t^{-|A|}[(P_{1}|x^{-}|y^{-}z|P_{2})_{xy}]\\ 
                           &+&
t^{|A|}[(P_{1}|xzy^{-}|P_{2})_y]\\
&+&
t^{-|A|}[(P_{1}|x^{-}yz|P_{2})_x]
\end{eqnarray*}
and
\begin{eqnarray*}
[P_{1}|BAx|CAyCBz|P_{2}] &=& t^{3|A|}[P_{1}|xyz|P_{2}]\\
                          &+&
(2t^{|A|}+t^{-3|A|}+t^{-|A|}(-t^{2}-t^{-2})
)[(P_{1}|z^{-}x^{-}y|P_{2})_{xz}]\\
                           &+&
t^{-|A|}[(P_{1}|x^{-}yz|P_{2})_x]\\ 
                           &+&
t^{-|A|}[(P_{1}|x^{-}|y^{-}z|P_{2})_{xy}]\\
&+&
t^{|A|}[(P_{1}|x^{-}y^{-}z|P_{2})_xy].
\end{eqnarray*}
Therefore, $[P_{1}|ABx|ACyBCz|P_{2}]$ is equal to $[P_{1}|BAx|CAyCBz|P_{2}]$.\par
Consider the third $S_1$-homotopy move
$$P_{1}|ABxACy|BCz|P_{2} \longrightarrow P_{1}|BAxCAy|CBz|P_{2}.$$
Then 
\begin{eqnarray*}
[P_{1}|ABxACy|BCz|P_{2}] &=& t^{3|A|}[P_{1}|xzy|P_{2}]\\
                          &+&
(2t^{|A|}+t^{-3|A|}+t^{-|A|}(-t^{2}-t^{-2})
)[(P_{1}|xy^{-}z|P_{2})_y]\\
                           &+&
t^{-|A|}[(P_{1}|x^{-}y^{-}z|P_{2})_{xy}]\\ 
                           &+&
t^{|A|}[(P_{1}|x^{-}yz|P_{2})_x]\\
&+&
t^{-|A|}[(P_{1}|yx^{-}|z|P_{2})_x]
\end{eqnarray*}
and
\begin{eqnarray*}
[P_{1}|BAxCAy|CBz|P_{2})] &=& t^{3|A|}[P_{1}|xyz|P_{2}]\\
                          &+&
(2t^{|A|}+t^{-3|A|}+t^{-|A|}(-t^{2}-t^{-2})
)[(P_{1}|x^{-}yz|P_{2})_{x}]\\
                           &+&
t^{-|A|}[(P_{1}|x^{-}y^{-}z|P_{2})_{xy}]\\ 
                           &+&
t^{-|A|}[(P_{1}|y^{-}x|z|P_{2})_{y}]\\
&+&
t^{|A|}[(P_{1}|xy^{-}z|P_{2})_y].
\end{eqnarray*}
Therefore, $[P_{1}|ABxACy|BCz|P_{2}]$ is equal to $[P_{1}|BAxCAy|CBz|P_{2}]$.\par
Consider the third $S_1$-homotopy move
$$P_{1}|ABx|ACy|BCz|P_{2} \longrightarrow P_{1}|BAx|CAy|CBz|P_{2}.$$
Then 
\begin{eqnarray*}
[P_{1}|ABx|ACy|BCz|P_{2}] &=& t^{3|A|}[P_{1}|y|zx|P_{2}]\\
                          &+&
(2t^{|A|}+t^{-3|A|}+t^{-|A|}(-t^{2}-t^{-2})
)[(P_{1}|y^{-}zx|P_{2})_y]\\
                           &+&
t^{-|A|}[(P_{1}|x^{-}|zy^{-}|P_{2})_{xy}]\\ 
                           &+&
t^{|A|}[(P_{1}|xzy^{-}|P_{2})_y]\\
&+&
t^{-|A|}[(P_{1}|yx^{-}|z|P_{2})_x]
\end{eqnarray*}
and
\begin{eqnarray*}
[P_{1}|BAx|CAy|CBz|P_{2}] &=& t^{3|A|}[P_{1}|yz|x|P_{2}]\\
                          &+&
(2t^{|A|}+t^{-3|A|}+t^{-|A|}(-t^{2}-t^{-2})
)[(P_{1}|y^{-}xz|P_{2})_y]\\
                           &+&
t^{-|A|}[(P_{1}|zy^{-}|x^{-}|P_{2})_{xy}]\\ 
                           &+&
t^{-|A|}[(P_{1}|x^{-}y|z|P_{2})_x]\\
&+&
t^{|A|}[(P_{1}|z^{-}yx^{-}|P_{2})_{xz}].
\end{eqnarray*}
Therefore, $[P_{1}|ABx|ACy|BCz|P_{2}]$ is equal to $[P_{1}|BAx|CAy|CBz|P_{2}]$.\par
The cases of $(|A|,$ $|B|,$ $|C|)$ $=$ $(\mp 1,$ $\pm 1,$ $\pm 1)$ and $(|A|,$ $|B|,$ $|C|)$ $=$ $(\pm 1,$ $\pm 1,$ $\mp 1)$ are proved in a similar way as the above cases.
\end{proof}
\begin{theorem}
For $(\mathcal{A}, P) \in \mathcal{P}_u (\alpha_1, \tau_1, S_1, \nu_1)$, we have an $S_1$-homotopy invariant by
\begin{equation}\label{jones}
J(P) = (-t)^{-3w(P)}\sum_{s:\text{states}} t^{\sigma(s)} (- t^{2} - t^{-2})^{|s|-1}, 
\end{equation}
where $w(P)$ $=$ $\sum_{A \in \mathcal{A}}|A|$.       
\end{theorem}
This formula (\ref{jones}) gives a definition of the Jones polynomial of $P \in \mathcal{P}_u (\alpha_1, \tau_1, S_1, \nu_1)$.

\begin{remark}\label{s1tos0}
Every $S_{1}$-homotopy invariant of pseudolinks is an $S_{0}$-homotopy invariant of nanophrases over $\alpha_{0}$ by using $S_{0}$ $\subset$ $S_{1}$ and the fact that the bijection $f : \alpha_0$ $\to$ $\alpha_1;$ $a \mapsto 1$ and $b \mapsto -1$ commutes  with the involutions, that is, $f \circ \tau_0$ $=$ $\tau_1 \circ f$ (cf.~Section~7.3, Remark~7.2 of \cite{Turaev2006}).  
\end{remark}
\begin{corollary}\label{jones-s0}
$J(P)$ is an $S_{0}$-homotopy invariant for nanophrases $P$ over $\alpha_{0}$.  
\end{corollary}

\section{A Jones polynomial of nanophrases over any {$\alpha$}}\label{sec6}
In Section~\ref{jones-by-word}, we discuss the $S_{1}$-homotopy invariant $J(P)$ of pseudolinks.  
In this section, we extend the Jones polynomial $J(P)$ to homotopy invariants of nanophrases over any $\alpha$. 

Let $\alpha$ be an alphabet, $\tau : \alpha \to \alpha$ an involution, and $\Delta_{\alpha}$ $=$ $\{(a, a, a)\}_{a \in \alpha}$.  We consider a complete system of representative $\{a_{1}, a_{2}, \dots, a_{m}\}$ of $\alpha/\tau$ and denote $\{a_{1}, a_{2}, \dots, a_{m}\}$ by ${\rm crs}(\alpha/\tau)$.  
 
\begin{definition}[\cite{Turaev2007}]\label{orbit}
For an alphabet $\alpha$ and an involution $\tau$, 
the {\it orbit} of the involution $\tau:$ $\alpha \to \alpha$ is a subset of $\alpha$ consisting either of one element preserved by $\tau$ or of two elements permuted by $\tau$; in the latter case, the orbit is {\it free}.  
\end{definition}
\begin{definition}
For $A \in \mathcal{A}$, we define the ${\rm sign}$ of $A$ by 
\begin{equation}
{\rm sign}_{L} (A) := \begin{cases}
                   & 1~\text{if}~|A| \in L ; \widetilde{|A|}: \text{a free orbit}\\
                   & -1~\text{if}~|A| \in \tau(L) ; \widetilde{|A|}: \text{a free orbit}\\
                   & 0~\ \ \text{otherwise}
                  \end{cases} 
\end{equation}
where $L$ is a nonempty subset of ${\rm crs}(\alpha/\tau)$. 
\end{definition}

Recall $\mathcal{P}(\alpha, \tau)$ that is a set of nanophrases over $\alpha$ with $\tau$.  

\begin{definition}
For any $(\alpha, \tau)$ and any subset $L \subset {\rm crs}(\alpha/\tau)$, the map 
$\mathcal{U}_{L}:$ $\mathcal{P}(\alpha, \tau) \to \mathcal{P}(\alpha_{0}, \tau_{0});$ $P \mapsto P_{0}$ is defined by the following two steps:

(Step 1) Remove $A \in \mathcal{A}$ such that ${\rm sign}_{L}(A) = 0$ from $(\mathcal{A}, P)$ $\in \mathcal{P}(\alpha, \tau)$.  

(Step 2) Let the nanophrase be $(\mathcal{A}', P')$ after removing letters from $(\mathcal{A}, P)$ by using (Step 1).  
For $\mathcal{A}'$, choose an appropriate alphabet $\mathcal{B}$ and consider a map $f$ sending $\mathcal{A}'$  to $\mathcal{B}$ such that
\begin{equation}
\begin{cases}
          &\text{transform}~$A$~\text{with}~{\rm sign}_{L}(A) = 1~\text{into}~B \in \mathcal{B}~\text{with}~|B| = 1,\\
          &\text{transform}~$A$~\text{with}~{\rm sign}_{L}(A) = -1~\text{into}~B \in \mathcal{B}~\text{with}~|B| = -1.  
       \end{cases}
\end{equation}
Then, $(\mathcal{B}, f \circ {P'})$ is denoted by $\mathcal{U}_L ((\mathcal{A}, P))$ or simply $\mathcal{U}_{L}(P)$.  
\end{definition}

\begin{theorem}\label{0_to_diagonal}
For any $L \subset {\rm crs}(\alpha/\tau)$ and for any nanophrases $(\mathcal{A}_{1}, P_{1})$ and $(\mathcal{A}_{2}, P_{2})$, 
if $(\mathcal{A}_{1}, P_{1}) \simeq_{\Delta_{\alpha}} (\mathcal{A}_{2}, P_{2})$, then $\mathcal{U}_{L}((\mathcal{A}_{1}, P_{1})) \simeq_{S_{0}} \mathcal{U}_{L}((\mathcal{A}_{2}, P_{2}))$.  
\end{theorem}
\begin{proof}
By definition, isomorphisms does not change the $\mathcal{U}_{L}(P)$.\par
Consider the first homotopy move 
$$P_{1}=(\mathcal{A}, xAAy) \longrightarrow P_{2}=(\mathcal{A} \setminus
\{A\}, xy)$$
where $x$ and $y$ are words on $\mathcal{A}$, 
possibly including the character  $|$.
Suppose ${\rm sign}(A) \neq 0$. Then,
$$\mathcal{U}_{L}(P_{1})=x_{L}AAy_{L} \simeq x_{L}y_{L} = 
\mathcal{U}_{L}(P_{2})$$
where $x_{L}$ and $y_{L}$ are words that are obtained by deleting all letters
$X \in \mathcal{A}$, such that ${\rm sign}_{L}(X) = 0$, from $x$ and
$y$, respectively.\par
Suppose ${\rm sign}(A) = 0$. Then,
$$\mathcal{U}_{L}(P_{1})=x_{L}y_{L}=\mathcal{U}_{L}(P_{2}).$$
Thus the first homotopy move does not change the 
homotopy class of $\mathcal{U}_{L}(P)$.\par
Consider the second homotopy move
$$P_{1}=(\mathcal{A}, xAByBAz) \longrightarrow P_{2} =
(\mathcal{A} \setminus \{A,B\}, xyz)$$
where $|A|=\tau(|B|)$, and $x$, $y$, and $z$ are words on $\mathcal{A}$
possibly including the character $|$.
Suppose $|A| \in L \cup \tau(L)$ and $\widetilde{|A|}$ is free orbit. 
Then, $|B| \in L \cup \tau(L)$ and $\widetilde{|A|}$ is free orbit 
since $|A|= \tau(|B|)$. Thus
$$\mathcal{U}_{L}(P_{1})=x_{L}ABy_{L}BAz_{L} 
\simeq x_{L}y_{L}z_{L} = \mathcal{U}_{L}(P_{2}).$$
where $x_{L}$, $y_{L}$, and $z_{L}$ are words 
that are obtained by deleting all letters
$X \in \mathcal{A}$ such that ${\rm sign}_{L}(X) = 0$, from $x$, $y$, and 
$z$, respectively.
Suppose $|A| \not\in L \cup \tau(L)$ or $|A|$ is a fixed point of $\tau$. 
Then, $|B| \not\in L \cup \tau(L)$ or $|B|$ is a fixed point of $\tau$ 
since $|A|= \tau(|B|)$. Thus,
$$\mathcal{U}_{L}(P_{1})= x_{L}y_{L}z_{L} = \mathcal{U}_{L}(P_{2}).$$
The above equation shows that 
the second homotopy move does not change the homotopy 
class of $\mathcal{U}_{L}(P)$.\par
Consider the third homotopy move 
$$P_{1}= (\mathcal{A}, xAByACzBCt) \rightarrow 
P_{2}=(\mathcal{A}, xBAyCAzCBt)$$
where $|A|=|B|=|C|$, and $x$, $y$, $z$, and $t$ are words on 
$\mathcal{A}$ possibly including the character $|$.
Suppose ${\rm sign}_{L}(A) \neq 0$. Then, ${\rm sign}(B) \neq 0$ and ${\rm sign}_{L}(C) \neq 0$ 
since $|A|=|B|=|C|$. Thus we obtain
$$\mathcal{U}_{L}(P_{1}) = x_{L}ABy_{L}ACz_{L}ACt_{L} 
\simeq x_{L}BAy_{L}CAz_{L}CBt_{L}
=\mathcal{U}_{L}(P_{2})$$
where $x_{L}$, $y_{L}$, $z_{L}$ and $t_{L}$ 
are words that are obtained by deleting all letters
$X \in \mathcal{A}$ such that ${\rm sign}_{L}(X) = 0$ from $x$, $y$, $z$, 
and $t$ respectively.
Suppose ${\rm sign}_{L}(A) = 0$. Then, ${\rm sign}(B), {\rm sign}_{L}(C) = 0$ 
since $|A|=|B|=|C|$. Thus we obtain
$$\mathcal{U}_{L}(P_{1}) = x_{L}y_{L}z_{L}t_{L} = \mathcal{U}_{L}(P_{2}).$$
Thus the third homotopy move does not change 
the homotopy class of $\mathcal{U}_{L}(P)$.\par
By the above argument,  
$\mathcal{U}_{L}$ is a homotopy invariant of nanophrases.   
\end{proof}

\begin{corollary}\label{f-cor}
Let $I$ be an $S_{0}$-homotopy invariant of nanophrase over $\alpha_{0}$.  
For $P$ $\in \mathcal{P}(\alpha, \tau)$, we define $I'$ as \[I'(P) := \big\{I(\mathcal{U}_{L}(P))\big\}_{L \subset {\rm crs(\alpha/\tau)}}.  \]
$I'$ is a $\Delta_{\alpha}$-homotopy invariant of $P$ $\in \mathcal{P}(\alpha, \tau)$.  
In particular, for $(\mathcal{A}, P)$ $\in \mathcal{P}(\alpha_{0}, \tau_{0})$, $I'(P)$ $=$ $\{I(P)\}$ if ${\rm crs}(\alpha_{0}/\tau_{0})$ $=$ $\{1\}$.  
\end{corollary}

Theorem \ref{0_to_diagonal} implies the following corollaries.  

\begin{corollary}\label{jones-apply}
Let $\alpha$ be an arbitrary alphabet.  $J(\mathcal{U}_{L}(P))$ are $\Delta_{\alpha}$-homotopy invariants for nanophrases $P$ over $\alpha$.  
\end{corollary}

\begin{remark}\label{fvl-to-delta}
Let $\alpha$ be an arbitrary alphabet.  For every invariant $i_{0}$ of pointed ordered flat virtual links, $i_{0}(\mathcal{U}_{L}(P))$ is a $\Delta_{\alpha}$-homotopy invariant for nanophrases $P$ over $\alpha$.  
\end{remark}

\section*{Acknowledgments}   
The first author was a Research Fellow of the Japan Society for the Promotion of Science,   and 
this work was partly supported by KAKENHI Grant Number 09J01599. 
The second author was a Research Fellow of the Japan Society for the Promotion of Science (Number:20$\cdot$935).  The work of the second author was partially supported by IRTG 1529, a fund of Waseda University (2010A-863), JSPS KAKENHI Grant Numbers JP20K03604, JPK22K03603, and Toyohashi Tech Project of Collaboration with KOSEN Grant Number 2309.  
The authors also thank Andrew Gibson who pointed out the possibility of a variation of $\mathcal{U}_{L}$ for a talk about an early version.
Finally, the authors would like to thank Professor Goo Ishikawa for his strong encouragement of the authors' research even before the evaluation of the development of the nanophrase theory was established.  
\bibliographystyle{plain}
\bibliography{Ref}

\begin{thebibliography}{1}

\bibitem{fukunaga2009homotopy}
Tomonori Fukunaga.
\newblock New invariants of nanophrases and homotopy classification of
  generalized phrases.
\newblock {\em Hokkaido Math. J.}, to appear.

\bibitem{Kauffman1987}
Louis~H. Kauffman.
\newblock State models and the {J}ones polynomial.
\newblock {\em Topology}, 26(3):395--407, 1987.

\bibitem{Manturov2010}
O.~V. Manturov and V.~O. Manturov.
\newblock Free knots and groups.
\newblock {\em J. Knot Theory Ramifications}, 19(2):181--186, 2010.

\bibitem{SilverWilliams2006}
Daniel~S. Silver and Susan~G. Williams.
\newblock An invariant for open virtual strings.
\newblock {\em J. Knot Theory Ramifications}, 15(2):143--152, 2006.

\bibitem{Turaev2005}
Vladimir Turaev.
\newblock Virtual strings.
\newblock {\em Ann. Inst. Fourier (Grenoble)}, 54(7):2455--2525 (2005), 2004.

\bibitem{Turaev2006}
Vladimir Turaev.
\newblock Knots and words.
\newblock {\em Int. Math. Res. Not.}, pages Art. ID 84098, 23, 2006.

\bibitem{Turaev2007Lec}
Vladimir Turaev.
\newblock Lectures on topology of words.
\newblock {\em Jpn. J. Math.}, 2(1):1--39, 2007.

\bibitem{Turaev2007}
Vladimir Turaev.
\newblock Topology of words.
\newblock {\em Proc. Lond. Math. Soc. (3)}, 95(2):360--412, 2007.

\bibitem{Viro2004}
Oleg Viro.
\newblock Khovanov homology, its definitions and ramifications.
\newblock {\em Fund. Math.}, 184:317--342, 2004.

\end{thebibliography}
\end{document}